\documentclass[11pt]{amsart}
\usepackage{amsmath,amssymb,amsthm,amscd}
\usepackage[a4paper,text={140mm,240mm},centering,headsep=5mm,footskip=10mm]{geometry}
\usepackage[matrix,arrow]{xy}
\numberwithin{equation}{section}

\theoremstyle{plain}
\newtheorem{theo}{Theorem}[section]
\newtheorem{lem}[theo]{Lemma}

\newtheorem{prop}[theo]{Proposition}

\newtheorem{coro}[theo]{Corollary}
\newtheorem{claim}[theo]{Claim}

\theoremstyle{definition}
\newtheorem{defi}[theo]{Definition}

\theoremstyle{remark}

\def\Zop{\bZ[1/p]}
\def\WnOX{W_n\cO_X}
\def\WO#1#2{W_{#1}\cO_{#2}}

\def\pzD{(\pz)_{X|D}}
\def\pnzD{(\pnz)_{X|D}}
\def\pXDp#1{(\bZ/p^{#1}\bZ)_{X|D/p}}

\def\pzE{(\pz)_{Y|E}}
\def\pnzE{(\pnz)_{Y|D}}
\def\pYEp#1{(\bZ/p^{#1}\bZ)_{Y|E/p}}
\def\pzX{(\pz)_X}

\def\ppzX#1#2{(\bZ/p^{#1}\bZ)_{X|#2}}
\newcommand{\Spec}{\mathrm{Spec}}

\newcommand{\Hom}{\text{\rm Hom}}

\newcommand{\Cone}{\mathrm{Cone}}

\def\art{{\mathrm{art}}} 
\def\rsw{{\mathrm{art}}}      
   
\def\art{{\mathrm{art}}}      

\def\ord{{\mathrm{ord}}}      
\def\Coker{{\mathrm{Coker}}}
\def\Ker{{\mathrm{Ker}}}    
\def\Hom{{\mathrm{Hom}}}    
\def\ord{{\mathrm{ord}}}    

\def\ch{{\mathrm{ch}}}    

\def\cF{{\mathcal F}}

\def\cK{{\mathcal K}}
\def\cL{{\mathcal L}}

\def\cO{{\mathcal O}}

\def\cK{{\mathcal K}}

\def\cL{{\mathcal{L}}}

\def\lam{\lambda}

\def\bZ{{\mathbb Z}}
\def\bQ{{\mathbb Q}}

\def\fm{{\frak{m}}}

\def\fmK{\fm_K}

\def\qz{{\bQ}/{\bZ}}

\def\pnz{\bZ/p^n\bZ}

\def\pz{\bZ/p\bZ}

\def\psz{\bZ/p^s\bZ}

\def\rmapo#1{\overset{#1}{\longrightarrow}}

\def\isom{\overset{\cong}{\longrightarrow}}

\def\mlam{m_\lam}
\def\Clam{C_\lam}

\def\Klam{K_{\lam}}

\def\fillog#1{\mathrm{fil}^{\log}_{#1}}
\def\fil#1{\mathrm{fil}_{#1}}

\def\gr#1{\mathrm{gr}_{#1}}

\def\FH#1#2{\fil{#2}H^1(#1)}

\def\FHlog#1#2{\fillog{#2}H^1(#1)}

\def\Zinf{Z_{\infty}}

\def\Zinf{Z_{\infty}}

\def\piabXD{\pi_1^{ab}(X,D)}

\def\piab#1{\pi^{ab}_1(#1)}

\def\qaq{\quad\text{and}\quad}
\def\qfor{\quad\text{for }}
\def\qwith{\quad\text{with }}

\def\K2#1{K_2(#1)}

\begin{document}

\title[Lefschetz theorem for abelian fundamental group with modulus]
{Lefschetz theorem for abelian fundamental group with modulus}

\author{Moritz Kerz}
\author{Shuji Saito}
\address{Moritz Kerz\\
NWF I-Mathematik\\
Universit\"at Regensburg\\
93040 Regensburg\\
Germany}
\email{moritz.kerz@mathematik.uni-regensburg.de}
\address{Shuji Saito\\
Interactive Research Center of Science, 
Graduate School of Science and Engineering,
 Tokyo Institute of Technology\\
Ookayama, Meguro\\
Tokyo 152-8551\\
Japan
}
\email{sshuji@msb.biglobe.ne.jp}

\begin{abstract}
We prove a Lefschetz hypersurface theorem for abelian fundamental groups allowing wild
ramification along some divisor.  In fact, we show that isomorphism holds if the degree
of the hypersurface is large relative to the ramification along the divisor.
\end{abstract}

\maketitle

\maketitle

\section{Statement of main results}\label{intro}

Lefschetz hyperplane theorems represent  an important technique in the study of
Grothen\-dieck's fundamental group $\pi_1(X)$ of an algebraic varieties $X$ (we omit base points for
simplicity). Roughly speaking one gets an isomorphism of the
form
\[
\iota_{Y/X} : \pi_1(Y) \xrightarrow{\sim} \pi_1(X)
\]
for a suitable hypersurface section $Y \to X$ if $\dim(X)  \ge 3$. Purely algebraic Lefschetz
theorems for projective varieties satisfying certain regularity assumptions were developed
in \cite{SGA2}. The case of non-proper varieties $X,Y$ is more intricate
because one needs a precise control of the ramification at the infinite locus. We show in
the present note that for the abelian quotient of the fundamental group a Lefschetz
hyperplane theorem does in fact hold. Our basic technical ingredient is the higher
dimensional ramification theory of Brylinski, Kato and Matsuda which is recalled in
Section~\ref{Lefschetz-RT}.
We expect that there is a non-commutative analog of our Lefschetz theorem, which should
have applications to $\ell$-adic representations of fundamental groups, especially over
finite fields as studied in \cite{EK}.

To formulate our main result, 
let $X$ be a normal variety over a perfect field $k$ and let $U\subset X$ be an open subset such that
$X\setminus U$ is the support of an effective Cartier divisor on $X$.
Let $D$ be an effective Cartier divisor on $X$ with support in $X\setminus U$.
We introduce the abelian fundamental group $\piabXD$ as a quotient of $\piab U$  
classifying abelian \'etale coverings of $U$ with ramification bounded by $D$. 
More precisely, for an integral curve $Z\subset U$, let $Z^N$ be the normalization of the closure of $Z$
in $X$ with $\phi_Z: Z^N \to X$, the natural map.
Let $\Zinf\subset Z^N$ be the finite set of points $x$ such that $\phi_Z(x)\not\in U$.
Then $\piabXD$ is defined as the Pontryagin dual of the group $\fil D H^1(U)$ of continuous characters 
$\chi:\piab U\to \qz$ such that for any integral curve $Z\subset U$ , 
its restriction $\chi|_{Z}: \piab {Z}\to \qz$ satisfies the following inequality of Cartier divisors on $Z^N$:
\[
\underset{y\in \Zinf}{\sum}\; \art_y(\chi|_{Z}) [y] \;\leq \;\phi_Z^* D,
\]
where $\art_y(\chi|_{Z})\in \bZ_{\ge 0}$ is the Artin conductor of $\chi|_{Z}$ at $y\in \Zinf$ and $\phi_Z^* D$
is the pullback of $D$ by the natural map $\phi_Z :Z^N \to X$. 

Such a global measure of ramification in terms of curves has been first considered by Deligne and Laumon, see \cite{La}.
\medbreak

Now assume that $X$ is smooth projective over $k$ (we fix a projective embedding) and that $C=X\setminus U$ is a simple normal crossing divisor. 
Let $Y$ be a smooth hypersurface section such that $Y\times_X C$ is a reduced simple normal crossing divisor on $Y$
and write $\deg (Y)$ for the degree of $Y$ with respect to the fixed projective embedding
of $X$.  Set $E=Y\times_X D$. Then one sees from the definition that
the map $Y\cap U \to U$ induces a natural map
\[
\iota_{Y/X}:  \piab {Y,E} \to \piab {X,D}.
\]

Our main theorem says:

\begin{theo}\label{thm.Lefschetz}
Assume that $Y$ is {\em sufficiently ample} with respect $(X,D)$ (see Definition \ref{def.ampleD}).
If $d:=\dim(X)\geq 3$, $\iota_{Y/X}$ is an isomorphism. 
If $d=2$, $\iota_{Y/X}$ is surjective. 
\end{theo}

The prime-to-$p$ part of the theorem is due to Schmidt and Spiess \cite{SS}, where $p=\ch(k)$.
Below we see that $Y$ is sufficiently ample if $\deg(Y)\gg 0$.

\begin{coro}\label{coro.finite}
Let $X$ be a normal proper variety over a finite field $k$.
Then $\piab {X,D}^0$ is finite, where
\[
\piab {X,D}^0=\Ker\big(\piab {X,D}\to \piab{\Spec(k)}\big).
\]
\end{coro}
\begin{proof}
In case $X$ and $X\setminus U$ satisfy the assumption of Theorem \ref{thm.Lefschetz}, 
the corollary follows from the corresponding statement for curves. The finiteness in the curves case is a consequence of class field theory.
For the general case, one can take by \cite{dJ} an alteration $f: X'\to X$ such that 
$X'$ and $X'\setminus U'$ with $U'=f^{-1}(U)$ satisfy the assumption of Theorem \ref{thm.Lefschetz}. 
Then the assertion follows from the fact that the map 
$f_*: \piab {U'} \to \piab U$ has a finite cokernel.
\end{proof}

Corollary~\ref{coro.finite} can also be deduced from  \cite[Thm.\ 6.2]{Ras}. It has recently
been generalized to the non-commutative setting by Deligne, see \cite{EK}.

Theorem~\ref{thm.Lefschetz} is a central ingredient in our paper \cite{KeS}. There we use
it to construct a reciprocity isomorphism between a Chow group of zero cycles with modulus
and the abelian fundamental group with bounded ramification. In fact
Theorem~\ref{thm.Lefschetz} allows us to restrict to surfaces in this construction.

\bigskip

\section{Review of ramification theory}\label{Lefschetz-RT}

First we review local ramification theory.
Let $K$ denote a henselian discrete valuation field of $\ch(K)=p>0$ with 
the ring $\cO_K$ of integers and residue field $\kappa$.
Let $\pi$ be a prime element of $\cO_K$ and $\fmK=(\pi)\subset \cO_K$ the maximal ideal.
By the Artin-Schreier-Witt theory, we have a natural isomorphism for $s\in \bZ_{\geq 1}$,
\begin{equation}\label{ASW.eq}
\delta_s: W_s(K)/(1-F)W_s(K) \isom H^1(K,\psz),
\end{equation}
where $W_s(K)$ is the ring of Witt vectors of length $s$ and $F$ is the Frobenius.
We have the Brylinski-Kato filtration indexed by integers $m\geq 0$
\[
\fillog m W_s(K) = \{(a_{s-1},\dots,a_1,a_0)\in W_s(K)\;|\; p^i v_K(a_i)\geq -m\},
\]
where $v_K$ is the normalized valuation of $K$. In this paper we use its non-log version introduced by 
Matsuda \cite{Ma}:
\[
\fil m W_s(K) = \fillog {m-1} W_s(K) + V^{s-s'} \fillog {m} W_{s'}(K),
\]
where $s'=\min\{s,\ord_p(m)\}$. 
We define ramification filtrations on $H^1(K):=H^1(K,\qz)$ as
\begin{align*}
 \FHlog K m &= H^1(K)\{p'\} \oplus \underset{s\geq1}{\cup} \delta_s(\fillog m W_s(K))
 &(m\ge 0),\\
\FH K m &= H^1(K)\{p'\} \oplus \underset{s\geq1}{\cup} \delta_s(\fil m W_s(K))
&(m\ge 1),
\end{align*}
where $H^1(K)\{p'\}$ is the prime-to-$p$ part of $H^1(K)$.
We note that this filtration is shifted by one from Matsuda's filtration
\cite[Def.3.1.1]{Ma}.
We also let $\FH {K} 0$ be the subgroup of all unramified characters.

\begin{defi}
For $\chi \in H^1(K)$ we denote the minimal $m$ with $\chi \in \FH{K} m$ by $\art_K(\chi)$
and call it the Artin conductor of $\chi$.
\end{defi}

We have the following fact (cf. \cite{Ka} and \cite{Ma}).

\begin{lem}\label{CFT.lem0} \mbox{}
\begin{itemize} 
\item[(1)]
$\FH {K} 1$ is the subgroup of tamely ramified characters.
\item[(2)]
$\FH {K} {m}\subset \FHlog {K} m \subset \FH {K} {m+1}$.
\item[(3)]
$\FH {K} m=\FHlog {K} {m-1}$ if $(m,p)=1$.
\end{itemize}
\end{lem}

The structure of graded quotients: 
\[
\gr m H^1(K)= \fil m H^1(K)/\fil {m-1} H^1(K) \quad( m>1)
\]
are described as follows. 
Let $\Omega_K^1$ be the absolute K\"ahler differential module and put
\[
\fil m \Omega_K^1 =\fm_K^{-m}\otimes_{\cO_K} \Omega_{\cO_K}^1.
\]
We have an isomorphism
\begin{equation}\label{grOmega}
\gr m \Omega_K^1 = \fil m \Omega_K^1/\fil {m-1} \Omega_K^1 \simeq 
\fmK^{-m} \Omega_{\cO_K}^1\otimes_{\cO_K} \kappa.
\end{equation}
We have the maps
\[
F^sd: W_s(K) \to \Omega_K^1\;;\; (a_{s-1},\dots,a_1,a_0) \to 
\underset{i=0}{\overset{s-1}{\sum}} \; a_i^{p^i-1} da_i.
\]
and one can check
$F^sd(\fil n W_s(K))\subset \fil n\Omega_K^1$. 

\begin{theo}\label{thm.Masuda}(\cite{Ma})
 The maps $F^sd$ factor through $\delta_s$ and induce a natural map
\[
\fil n  H^1(K) \to \fil n \Omega_K^1
\]
which induces for $m>1$ an injective map (called the refined Artin conductor for $K$)
\begin{equation}\label{rswK}
\rsw_K: \gr n  H^1(K) \hookrightarrow \gr n \Omega_K^1.
\end{equation}
\end{theo}
\bigskip

Next we review global ramification theory.
Let $X,C$ be as in the introduction and fix a Cartier divisor $D$ with $|D| \subset C$.
We recall the definition of $\piabXD$.
We write $H^1(U)$ for the \'etale cohomology group $H^1(U,\qz)$ which is identified with 
the group of continuous characters $\piab U\to \qz$.

\begin{defi} \label{gloram.def1}
We define $\fil D H^1(U)$ to be the subgroup of $\chi\in H^1(U)$ satisfying the condition:
for all integral curves $Z\subset X$ not contained in $C$, 
its restriction $\chi|_{Z}: \piab Z\to \qz$ satisfies the following inequality of Cartier divisors on $Z^N$:
\[
\underset{y\in \Zinf}{\sum}\; \art_y(\chi|_{Z}) [y] \;\leq \;\phi_Z^* D,
\]
where $\art_y(\chi|_{Z})\in \bZ_{\ge 0}$ is the Artin conductor of $\chi|_{Z}$ at $y\in \Zinf$ and $\phi_Z^* D$
is the pullback of $D$ by the natural map $\phi_Z :Z^N \to X$. 
Define
\begin{equation}\label{piabXD}
\piabXD=\Hom(\FH U D,\qz),
\end{equation}
endowed with the usual pro-finite topology of the dual.
\end{defi}

\medskip

For the rest of this section we assume that $X$ is smooth and $C$ is simple normal crossing.
Let $I$ be the set of generic points of $C$ and let $\Clam=\overline{\{\lam\}}$ for $\lam\in
I$.
Write
 \begin{equation}\label{eqD}
 \displaystyle{D=\underset{\lambda\in I}{\sum} \mlam \Clam}.
 \end{equation}
For $\lam\in I$, let $\Klam$ be the henselization of $K=k(X)$ at $\lam$.
Note that $\Klam$ is a henselian discrete valuation field with residue field $k(\Clam)$. 

\begin{prop}\label{gloram.prop1}
We have 
\[
\fil D H^1(U) = \Ker\big(H^1(U) \to \underset{\lam\in I}{\bigoplus}\; H^1(\Klam)/\fil {\mlam} H^1(\Klam)\big).
\]
\end{prop}
\begin{proof}
This is a consequence of ramification theory developed in \cite{Ka} and \cite{Ma}.
See \cite[Cor.2.7]{KeS} for a proof.
\end{proof}
\medbreak

\begin{prop}\label{gloram.prop1}
Fix $\lam\in I$ such that $m_\lam >1$ in \eqref{eqD}.
The refined Artin conductor $\rsw_{\Klam}$ (cf. Theorem \ref{thm.Masuda}) induces
a natural injective map
\[
\rsw_{\Clam}: \fil D H^1(U)/\fil {D-\Clam} H^1(U) \hookrightarrow
 H^0(\Clam,\Omega^1_X(D)\otimes_{\cO_X} \cO_{\Clam})
\]
which is compatible with pullback along maps $f:X' \to X$ of smooth varieties with the property that $f^{-1}(C)$
is a reduced simple normal crossing divisor.
\end{prop}
\begin{proof}
This follows from the integrality result \cite[4.2.2]{Ma} of the refined Artin conductor.
\end{proof}

\medbreak
Proposition \ref{gloram.prop1} motivates us to introduce the following log-variant of $\fil D H^1(U)$.

\begin{defi} \label{gloram.def2}
We define $\fillog D H^1(U)$ as
\[
\fillog D H^1(U) = \Ker\big(H^1(U) \to \underset{\lam\in I}{\bigoplus}\; H^1(\Klam)/\fillog {\mlam} H^1(\Klam)\big).
\]
\end{defi}

\begin{lem}\label{CFT.lem1} \mbox{}
\begin{itemize}
\item[(1)]
$\fil C H^1(U)$ is the subgroup of tamely ramified characters.
\item[(2)]
$\fil D H^1(U)\subset \fillog D H^1(U) \subset \fil {D+C} H^1(U)$.
\item[(3)]
$\fil D H^1(U)=\fillog {D-C} H^1(U)$ if $(\mlam,p)=1$ for all $\lam\in I$.
\end{itemize}
\end{lem}
\begin{proof}
This is a direct consequence of Lemma \ref{CFT.lem0}.
\end{proof}

\bigskip

\section{Proof of the main theorem}\label{Lefschetz-proof}


Let $X$ be a smooth projective variety over a perfect field of characteristic $p>0$ 
and $C\subset X$ a reduced simple normal crossing divisor on $X$.
Let $j:U=X\setminus C\subset X$ be the open immersion.
We use the same notation as in the last part of the previous section.
Take an effective Cartier divisor 
\[\displaystyle{D=\underset{\lambda\in I}{\sum} \mlam \Clam \quad\text{with }\mlam\geq 0}.\]
Let $I'=\{\lam\in I\;|\; p|\mlam\}$ and put
\[\displaystyle{D'=\underset{\lambda\in I'}{\sum} (\mlam+1) \Clam \;+ \;
\underset{\lambda\in I\setminus I'}{\sum} \mlam \Clam}.\]

Let $Y$ be a smooth hypersurface section such that $Y\times_X C$ is a reduced simple normal crossing divisor on $Y$.

\begin{defi}\label{def.ampleD} \mbox{}
\begin{itemize}
\item[(1)]
Assuming $\dim(X)\geq 3$, we say that $Y$ is sufficiently ample for $(X,D)$ if the following conditions hold:
\begin{itemize}
\item[$(A1)$]
$H^i(X,\Omega^d_X(-\Xi+Y))=0$ for any effective Cartier divisor $\Xi \leq D$ and for $i=d,d-1,d-2$.
\item[$(A2)$]
For any $\lam\in I'$, we have
\[
H^0(\Clam,\Omega^{1}_X(D'-Y)\otimes\cO_{\Clam})=H^0(\Clam,\cO_{\Clam}(D'-Y))=H^1(\Clam,\cO_{\Clam}(D'-2Y))=0.
\]
\end{itemize}
\item[(2)]
Assuming $\dim(X)=2$, we say that $Y$ is sufficiently ample for $(X,D)$ if the following condition holds:
\begin{itemize}
\item[$(B)$]
$H^i(X,\Omega^d_X(-\Xi+Y))=0$ for any effective Cartier divisor $\Xi \leq D$ and for $i=1,2$.
\end{itemize}
\end{itemize}
\end{defi}

We remark that there is an integer $N$ such that any smooth $Y$ of degree$\geq N$
is sufficiently ample for $(X,D)$.

\medbreak

Theorem \ref{thm.Lefschetz} is a direct consequence of the following.

\begin{theo}\label{thm.Lefschetz.dual}
Let $Y$ be sufficiently ample for $(X,D)$. Write $E=Y\times_X D$. 
\begin{itemize}
\item[(1)]
Assuming $d:=\dim(X)\geq 3$, we have isomorphisms
\[
\fil D H^1(U) \isom \fil E H^1(U\cap Y)\qaq \fillog D H^1(U) \isom \fillog E H^1(U\cap Y).
\]
\item[(2)]
Assuming $d=2$, we have injections
\[
\fil D H^1(U) \hookrightarrow \fil E H^1(U\cap Y) \qaq \fillog D H^1(U) \hookrightarrow \fillog E H^1(U\cap Y).
\]
\end{itemize}
\end{theo}

For an abelian group $M$, we let $M\{p'\}$ denote the prime-to-$p$ torsion part of $M$.
 
\begin{lem}\label{Lefschetz.tame} 
\begin{itemize}
\item[(1)]
Assuming $d:=\dim(X)\geq 3$, we have an isomorphism
\[
\fil D H^1(U)\{p'\} \isom \fil E H^1(U\cap Y)\{p'\}
\]
and the same isomorphism for $\fillog D$.
\item[(2)]
Assuming $d=2$, we have an injection
\[
\fil D H^1(U)\{p'\} \hookrightarrow \fil E H^1(U\cap Y)\{p'\}
\]
and the same injection for $\fillog D$.
\end{itemize}
\end{lem}
\begin{proof}
Noting 
\[
\fil D H^1(U)\{p'\} =\fil C H^1(U)\{p'\} = \fillog C H^1(U)\{p'\} =\fillog D H^1(U)\{p'\},
\]
this follows from the tame case of Theorem \ref{thm.Lefschetz} due to \cite{SS}.
\end{proof}
\medbreak

By the above lemma, Theorem \ref{thm.Lefschetz.dual} is reduced to the following.

\begin{theo}\label{thm.Lefschetz.dual2}
Let the assumption be as in Theorem~\ref{thm.Lefschetz.dual}. Take an integer $n>0$.
\begin{itemize}
\item[(1)]
Assuming $d:=\dim(X)\geq 3$, we have isomorphisms
\[
\fil D H^1(U)[p^n] \isom \fil E H^1(U\cap Y)[p^n]
\]
and the same isomorphism for $\fillog D$.
\item[(2)]
Assuming $d=2$, we have an injection
\[
\fil D H^1(U)[p^n] \hookrightarrow \fil E H^1(U\cap Y)[p^n]
\]
and the same injection for $\fillog D$.
\end{itemize}
\end{theo}

\bigskip

In what follows we consider an effective Cartier divisor with $\Zop$-coefficient:
\[
D=\underset{\lambda\in I}{\sum} \mlam \Clam,\quad \mlam\in \Zop_{\geq 0}.
\]
We put
\[
[D]=\underset{\lambda\in I}{\sum} [\mlam] \Clam\qwith
[\mlam]=\max\{i\in \bZ\;|\; i\leq \mlam \}
\]
and $\cF(\pm D)=\cF\otimes_{\cO_X} \cO_X(\pm [D])$ for an $\cO_X$-module.
For $D$ as above, let $\fillog D \WnOX$ be the subsheaf of $j_*\WO n U$ of local sections
\[\underline{a} \in \WO n U \quad\text{such that }
\underline{a} \in \fillog {\mlam} W_n(\Klam)\qfor \text{any }\lam\in I,
\]
where $\fillog {\mlam} W_n(\Klam):=\fillog {[\mlam]} W_n(\Klam)$ is defined in \S\ref{Lefschetz-RT} for the henselization $\Klam$ of $K=k(X)$ at $\lam$. 
We note
\[
\cO_X(D)=\fillog D \WnOX \qfor n=1.
\]
The following facts are easily checked:
\begin{itemize}
\item
The Frobenius $F$ induces $F:\fillog {D/p} \WnOX\to \fillog D \WnOX$.
\item
The Verschiebung $V$ induces $V:\fillog {D} \WO {n-1} X \to \fillog D \WnOX$.
\item
The restriction $R$ induces $R:\fillog {D} \WnOX \to \fillog {D/p} \WO {n-1} X$.
\item
The following sequence is exact:
\begin{equation}\label{eq.red}
0\to \cO_X(D) \rmapo{V^{n-1}} \fillog {D} \WnOX \rmapo{R} \fillog {D/p} \WO {n-1} X \to 0.
\end{equation}
\end{itemize}

We define an object $\pnzD$ of the derived category $D^b(X)$ of bounded complexes of \'etale sheaves on $X$:
\[
\pnzD = \Cone\big(\fillog {D/p} \WnOX \rmapo{1-F} \fillog D \WnOX\big)[-1].
\]
We have a distinguished triangle in $D^b(X)$:
\begin{equation}\label{pnzDdt}
\pnzD \to \fillog {D/p} \WnOX \rmapo{1-F} \fillog D \WnOX \rmapo{+}.
\end{equation}

\begin{lem}\label{lem0.FHp}
There is a distinguished triangle
\[
\pzD \to \pnzD \to \ppzX {n-1} {D/p} \rmapo{+}.
\]
\end{lem}
\begin{proof}
The lemma follows from the commutative diagram
\[
\xymatrix{
0\ar[r] & \cO_X(D/p) \ar[r]^{V^{n-1}}\ar[d]^{1-F} & \fillog {D/p} \WnOX
\ar[r]^(0.45){R}\ar[d]^{1-F} & \fillog {D/p^2} \WO {n-1} X \ar[d]^{1-F} \ar[r] &0\\
0\ar[r] & \cO_X(D) \ar[r]^{V^{n-1}}  & \fillog {D} \WnOX \ar[r]^(0.45){R}  & \fillog {D/p} \WO
{n-1} X \ar[r] & 0\\
}
\]
\end{proof}

\medbreak

\begin{lem}\label{lem.FHp}
There is a canonical isomorphism
\[
\fillog D H^1(U)[p^n] \simeq H^1(X,\pnzD).
\]
\end{lem}
\begin{proof}
Noting that the restriction of $\pnzD$ to $U$ is $\pnz$ on $U$, we have the localization exact sequence
\begin{equation}\label{pzpls}
 H^1(X,\pnzD) \to H^1(U,\pnz) \to  H^2_C(X,\pnzD).
\end{equation}
For the generic point $\lam$ of $\Clam$, \eqref{pnzDdt} gives us an exact sequence
\[
H^1_\lam(X,\fillog {D/p}\WnOX) \rmapo{1-F} H^1_\lam(X,\fillog D\WnOX) \to H^2_\lam(X,\pnzD) \to 
H^2_\lam(X,\fillog {D/p} \WnOX).
\]
By \cite[Cor.3.10]{Gr} and \eqref{eq.red} we have
\[
H^i_\lam(X,\fillog {D/p}\WnOX)= H^i_\lam(X,\fillog {D}\WnOX)= 0\qfor i\geq 2
\]
and 
\[
\begin{aligned}
&H^1_\lam(X,\fillog {D/p}\WnOX) \simeq W_n(\Klam)/\fillog {\mlam/p} W_n(\Klam), \\
&H^1_\lam(X,\fillog {D}\WnOX) \simeq W_n(\Klam)/\fillog {\mlam} W_n(\Klam).
\end{aligned}
\]
Thus we get 
\[
 H^2_\lam(X,\pnzD) \simeq H^1(\Klam)[p^n]/\fillog {\mlam} H^1(\Klam)[p^n].
\]
Hence Lemma \ref{lem.FHp} follows from \eqref{pzpls} and the injectivity of 
\[
H^2_C(X,\pnzD) \to \underset{\lam\in I}{\bigoplus} \; H^2_\lam(X,\pnzD).
\]

This injectivity is a consequence of

\begin{claim}\label{pnzDdt.eq}
For $x\in C$ with $\dim(\cO_{X,x})\geq 2$ we have
\begin{equation*}
H^2_x(X,\pnzD)=0.
\end{equation*}
\end{claim}

By Lemma \ref{lem0.FHp} it suffices to show Claim~\ref{pnzDdt.eq} in case $n=1$.
Triangle~\eqref{pnzDdt} gives us an exact sequence
\[
H^1_x(X,\cO_X(D)) \to H^2_x(X,\pzD) \to H^2_x(X,\cO_X(D/p)) \rmapo{1-F} H^2_x(X,\cO_X(D)).
\]
If $\dim(\cO_{X,x})> 2$, $H^1_x(X,\cO_X(D))=0$ and $H^2_x(X,\cO_X(D/p))=0$ by \cite[Cor.3.10]{Gr}, 
which implies $H^2_x(X,\pzD)=0$ as desired. 

We now assume $\dim(\cO_{X,x})=2$. 
Let $\pzX$ denote the constant sheaf $\pz$ on $X$ and put
\[
\cF_{X|D} =\Coker\big(\cO_X(D/p) \rmapo{1-F} \cO_X(D)\big).
\]
Note that $\cF_{X|D}=0$ for $D=0$.
By definition we have a distinguished triangle
\[
\pzX \to \pzD \to \cF_{X|D}  \rmapo{+} .
\]
By \cite[X, Theorem 3.1]{SGA1}, we have $H^2_x(X,\pzX)=0$.
Hence we are reduced to showing 
\begin{equation}\label{FXD}
H^2_x(X,\cF_{X|D})=0.
\end{equation}
Without loss of generality we can assume that $D$ has integral coefficients.
We prove \eqref{FXD} by induction on multiplicities of $D$ reducing to the case $D=0$.
Fix an irreducible component $\Clam$ of $C$ with the multiplicity $\mlam\geq 1$ in $D$ and put $D'=D-\Clam$. 
We have a commutative diagram with exact rows and columns
\[
\xymatrix{
& \pzX \ar[d] & \pzX \ar[d] \\
0\ar[r] & \cO_X(D'/p) \ar[r]\ar[d]^{1-F} & \cO_X(D/p) \ar[r]\ar[d]^{1-F} & \cL  \ar[d]^{F} \ar[r] &0\\
0\ar[r] & \cO_X(D') \ar[r]  & \cO_X(D) \ar[r]  & \cO_{\Clam}(D) \ar[r] & 0\;.\\
}
\]
Here $\cO_{\Clam}(D)=\cO_{X}(D)\otimes \cO_{\Clam}$, and $\cL=\cO_{\Clam}(D/p)$ if $p|\mlam$, and $\cL=0$ otherwise. 
Thus we get short exact sequences
\[
0\to \cF_{X|D'} \to \cF_{X|D} \to \cO_{\Clam}(D) \to 0 \quad\text{if }p\not|\mlam,
\]
\[
0\to \cF_{X|D'} \to \cF_{X|D} \to \cO_{\Clam}(D)/\cO_{\Clam}(D/p)^p \to 0 \quad\text{if }p|\mlam.
\]
We may assume $H^2_x(X,\cF_{X|D'})=0$ by the induction hypothesis.
Hence \eqref{FXD} follows from 
\begin{equation}\label{FXD2-2}
H^2_x(\Clam,\cO_{\Clam}(D))=0,
\end{equation}
\begin{equation}\label{FXD2-3}
H^2_x(\Clam,\cO_{\Clam}(D)/\cO_{\Clam}(E)^p)=0,
\end{equation}
where we put $E=[D/p]$.
We may assume $x\in \Clam$ so that $\dim(\cO_{\Clam,x})=1$ by the assumption $\dim(\cO_{X,x})=2$.
\eqref{FXD2-2} is a consequence of \cite[Cor.3.10]{Gr}.
In view of an exact sequence
\[
0\ \to \cO_{\Clam}(p E)/\cO_{\Clam}(E)^p \to \cO_{\Clam}(D)/\cO_{\Clam}(E)^p \to \cO_{\Clam}(D)/\cO_{\Clam}(p E)
\to 0\;,
\]
\eqref{FXD2-3} follows from
\[
H^2_x(\Clam,\cO_{\Clam}(p E)/\cO_{\Clam}(E)^p)=0\qaq
H^2_x(\Clam,\cO_{\Clam}(D)/\cO_{\Clam}(p E))=0.
\]
The first assertion follows from \cite[Cor.3.10]{Gr} noting that
$\cO_{\Clam}(p E)/\cO_{\Clam}(E)^p$ is a locally free $\cO_{\Clam}^p$-module. 
The second assertion holds since $\cO_{\Clam}(D)/\cO_{\Clam}(p E)$ is supported
in a proper closed subscheme $T$ of $\Clam$ and $x$ is a generic point of $T$ if $x\in T$.
This completes the proof of Lemma \ref{lem.FHp}.
$\square$
\end{proof}
\bigskip

In view of the above results, the assertions for $\fillog D$ of Theorem \ref{thm.Lefschetz.dual2}(1) and (2) 
follows from the following.

\begin{theo}\label{lem2.Lefschetz}
Let the assumption be as in Theorem \ref{thm.Lefschetz.dual}. The natural map
\[
H^1(X,\pnzD) \to H^1(Y,\pnzE)
\]
is an isomorphism for $d:=\dim(X)\geq 3$, and it is injective for $d=2$.
\end{theo}
\begin{proof}
By Lemma~\ref{lem0.FHp} we have a commutative diagram:
\[
\xymatrix{
 0 \ar[d] & 0 \ar[d]\\
H^1(X,\pzD) \ar[r] \ar[d]& H^1(Y,\pzE) \ar[d]\\
H^1(X,\pnzD) \ar[r]\ar[d] & H^1(Y,\pnzE) \ar[d] \\
H^1(X,\pXDp {n-1}) \ar[r]\ar[d] & H^1(Y,\pYEp {n-1}) \ar[d] \\ 
H^2(X,\pzD) \ar[r] & H^2(Y,\pzE) \\
}
\] 
The theorem follows by the induction on $n$ from the following.
\end{proof}

\begin{lem}\label{lem1.Lefschetz}
Let the assumption be as Theorem \ref{thm.Lefschetz.dual}.
\begin{itemize}
\item[(1)]
Assuming $d\geq 3$, the natural map 
\[
H^i(X,\pzD) \to H^i(Y,\pzE)
\]
is an isomorphism for $i=1$ and injective for $i=2$.
\item[(2)]
Assuming $d=2$, the natural map 
\[
H^1(X,\pzD) \to H^1(Y,\pzE)
\]
is injective.
\end{itemize}
\end{lem}
\begin{proof}
We define an object $\cK$ of $D^b(X)$:
\[
\cK = \Cone\big(\cO_X(D/p-Y) \rmapo{1-F} \cO_X(D-Y) \big)[-1].
\]
By the commutative diagram with exact horizontal sequences:
\[\xymatrix{
 0 \ar[r] & \cO_X(D/p-Y) \ar[r]\ar[d]^{1-F} & \cO_X(D/p) \ar[r]\ar[d]^{1-F} & \cO_Y(E/p) \ar[r]\ar[d]^{1-F} & 0 \\
 0 \ar[r] & \cO_X(D-Y) \ar[r] & \cO_X(D) \ar[r] & \cO_Y(E) \ar[r] & 0 \\
}\]
we have a distinguished triangle in $D^b(X)$:
\[
\cK \to \pzD  \to \pzE \rmapo{+}.
\]
Hence it suffices to show $H^i(X,\cK)=0$ for $i=1,2$ in case $d\geq 3$ and $H^1(X,\cK)=0$ in case $d=2$. 
We have an exact sequence
\begin{align*}
H^0(\cO_X(D-Y)) &\to H^1(X,\cK) \to H^1(\cO_X(D/p-Y))\\
&  \to H^1(\cO_X(D-Y)) 
 \to H^2(X,\cK) \to H^2(\cO_X(D/p-Y))\\
\end{align*}
By Serre duality, for a divisor $\Xi$ on $X$, we have
\[
H^i(X,\cO_X(\Xi-Y)) = H^{d-i}(X,\Omega^d_X(-\Xi+Y))^\vee.
\]
Thus the desired assertion follows from Definition \ref{def.ampleD}$(A1)$ and $(B)$.
\end{proof}
\bigskip

It remains to deduce the assertions for $\fil D$ of Theorem \ref{thm.Lefschetz.dual2}(1) and (2) 
from that for $\fillog D$.
Let $D'$ be as in the beginning of this section and $E'=D'\times_XY$. 
Noting that the multiplicities of $D'$ are prime to $p$, we have by Lemma \ref{CFT.lem1}(3) 
\[
\fil {D'} H^1(U)=\fillog {D'-C} H^1(U) \qaq
\fil {E'} H^1(U\cap Y)=\fillog {E'-C\cap Y} H^1(U\cap Y).
\]
Thus the assertions for $\fillog {D'-C}$ of Theorem \ref{thm.Lefschetz.dual2} implies
that for $\fil {D'}$. Since $\fil D\subset \fil {D'}$, it immediately implies the injectivity of
\[
\fil {D} H^1(U)\to  \fil {E} H^1(U\cap Y).
\]
It remains to deduce its surjectivity from that of 
\[
\fil {D'} H^1(U)\to  \fil {E'} H^1(U\cap Y)
\]
assuming $d\geq 3$. For this it suffices to show the injectivity of 
\[
\fil {D'} H^1(U)/\fil {D} H^1(U) \to \fil {E'} H^1(U\cap Y)/\fil {E} H^1(U\cap Y).
\]
By Proposition \ref{gloram.prop1} we have a commutative diagram
\[\xymatrix{
\fil {D'} H^1(U)/\fil {D} H^1(U) \ar@{^{(}->}[r] \ar[d] & 
\underset{\lam\in I'}{\bigoplus}\; H^0(\Clam,\Omega^1_X(D')\otimes_{\cO_X} \cO_{\Clam}) \ar[d] \\
\fil {E'} H^1(U\cap Y)/\fil {E} H^1(U\cap Y) \ar@{^{(}->}[r]  & 
\underset{\lam\in I'}{\bigoplus}\; H^0(\Clam\cap Y,\Omega^1_Y(D')\otimes_{\cO_Y} \cO_{\Clam\cap Y}) \\
}\]
Thus we are reduced to showing the injectivity of the right vertical map.
Putting $\cL= \Ker(\Omega^1_X \to i_*\Omega^1_Y)$ where $i:Y\subset X$, the assertion follows from
\[
 H^0(\Clam,\cL(D')\otimes_{\cO_X} \cO_{\Clam})=0.
\]
Note that we used the fact that $Y$ and $\Clam$ intersect transversally.
We have an exact sequence
\[
0\to \Omega^1_X(-Y) \to \cL \to \cO_X(-Y)\otimes \cO_Y \to 0.
\]
From this we get an exact sequence
\[
0\to \Omega^1_X(D'-Y)\otimes_{\cO_X} \cO_{\Clam} \to \cL(D')\otimes_{\cO_X} \cO_{\Clam} \to
\cO_{\Clam}(D'-Y)\otimes \cO_{\Clam\cap Y} \to 0.
\]
We also have an exact sequence
\[
0\to \cO_{\Clam}(D'-2Y)\to \cO_{\Clam}(D'-Y) \to \cO_{\Clam}(D'-Y)\otimes \cO_{\Clam\cap Y} \to 0.
\]
Therefore the desired assertion follows from Definition \ref{def.ampleD}$(A2)$. 
This completes the proof of Theorem \ref{thm.Lefschetz.dual2}.
$\square$

\bigskip

\end{document}